\documentclass[preprint,3p,10pt,authoryear]{elsarticle}

\journal{Journal of Mathematical Analysis and Applications}

\usepackage{amssymb}
\usepackage{amsthm}
\usepackage{lineno}  

\usepackage{enumerate,graphicx,mathtools,natbib,geometry,xcolor}
\usepackage[colorlinks]{hyperref}

\newtheorem{theorem}{Theorem}[section]
\newtheorem{proposition}[theorem]{Proposition}

\newtheorem{lemma}[theorem]{Lemma}
\newtheorem{remark}[theorem]{Remark}

\makeatletter \@addtoreset{equation}{section} \makeatother

\newcommand{\N}{\mathbb{N}}
\newcommand{\Z}{\mathbb{Z}}
\newcommand{\R}{\mathbb{R}}
\newcommand{\C}{\mathbb{C}}

\newcommand{\PP}{\mathbb{P}}
\newcommand{\EE}{\mathbb{E}}

\newcommand{\prob}[1]{\mathbb{P}\hspace{-0.5mm}\left(#1\right)}
\newcommand{\esp}[1]{\mathbb{E}\hspace{-0.5mm}\left[#1\right]}

\begin{document}

\begin{frontmatter}

    \title{Large deviations and continuity estimates for the derivative \\of a random model of \texorpdfstring{$\log |\zeta|$}{log zeta} on the critical line}

    \author[a1]{Louis-Pierre Arguin\fnref{fn1}}
    \author[a2]{Fr\'ed\'eric Ouimet\corref{cor2}\fnref{fn2}}

    \address[a1]{Baruch College and Graduate Center (CUNY), New York, NY 10010, USA.}
    \address[a2]{Universit\'e de Montr\'eal, Montr\'eal, QC H3T 1J4, Canada.}

    \cortext[cor2]{Corresponding author}
    \ead{ouimetfr@dms.umontreal.ca}

    \fntext[fn1]{L.-P. Arguin is supported in part by NSF Grant DMS-1513441 and by NSF CAREER DMS-1653602.}
    \fntext[fn2]{F. Ouimet is supported by a NSERC Doctoral Program Alexander Graham Bell scholarship (CGS D3).}

    \begin{abstract}
        In this paper, we study the random field
        \begin{equation*}
            X(h) \circeq \sum_{p \leq T} \frac{\text{Re}(U_p \, p^{-i h})}{p^{1/2}}, \quad h\in [0,1],
        \end{equation*}
        where $(U_p, \, p ~\text{primes})$ is an i.i.d.\hspace{-0.3mm} sequence of uniform random variables on the unit circle in $\C$.
        \cite{arXiv:1304.0677} showed that $(X(h), \, h\in (0,1))$ is a good model for the large values of $(\log |\zeta(\frac{1}{2} + i (T + h))|, \, h\in [0,1])$ when $T$ is large, if we assume the Riemann hypothesis.
        The asymptotics of the maximum were found in \cite{MR3619786} up to the second order, but the tightness of the recentered maximum is still an open problem.
        As a first step, we provide large deviation estimates and continuity estimates for the field's derivative $X'(h)$.
        The main result shows that, with probability arbitrarily close to $1$,
        \begin{equation*}
            \max_{h\in [0,1]} X(h) - \max_{h\in \mathcal{S}} X(h) = O(1),
        \end{equation*}
        where $\mathcal{S}$ a discrete set containing $O(\log T \sqrt{\log \log T})$ points.
    \end{abstract}


\end{frontmatter}

\section{Introduction}\label{sec:intro}

    In \cite{FyodorovHiaryKeating2012} and \cite{MR3151088}, it was conjectured that if $\tau$ is sampled uniformly in $[T,2T]$ for some large $T$, then the law of the maximum of $(\log |\zeta(\frac{1}{2} + i(\tau + h))|, h\in [0,1])$, where $\zeta$ denotes the Riemann zeta function, should be asymptotic to $\log \log T - \frac{3}{4} \log \log \log T + \mathcal{M}_T$ where $(\mathcal{M}_T, T\geq 2)$ is a sequence of random variables converging in distribution.
    At present, the first order of the maximum is proved conditionally on the Riemann hypothesis in \cite{MR3851835} and unconditionally in \cite{arXiv:1612.08575}.

    In order to study this hard problem originally, a randomized version of the Riemann zeta function was introduced in \cite{arXiv:1304.0677}, see \eqref{def:X}.
    The first order of the maximum was proved in \cite{arXiv:1304.0677}, the second order of the maximum was proved in \cite{MR3619786}, and a related study of the Gibbs measure can be found in \cite{arXiv:1706.08462} and \cite{MR3841407}.
    The tightness of the recentered maximum is still open.

    As a first step, our main result (Theorem \ref{thm:prop:large.deviation.estimates.derivative}) shows that the tightness of the ``continuous'' maximum $\max_{h\in [0,1]} X(h)$ (once recentered) can be reduced to the tightness of a ``discrete'' maximum $\max_{h\in \mathcal{S}} X(h)$ (once recentered) where $\mathcal{S}$ is a discrete set containing $O(\log T \sqrt{\log\log T})$ points.
    In order to prove Theorem \ref{thm:prop:large.deviation.estimates.derivative}, we will need continuity estimates and large deviation estimates for the field's derivative $X'(h)$, which can be found in Proposition \ref{prop:continuity.estimates.derivative} and Proposition \ref{prop:large.deviation.estimates.derivative}, respectively.

    The paper is organised as follows.
    In Section \ref{sec:model}, we introduce the model $X(h)$.
    In Section \ref{sec:main.result}, the main result is stated and proven.
    Proposition \ref{prop:continuity.estimates.derivative} and Proposition \ref{prop:large.deviation.estimates.derivative} are stated in Section \ref{sec:main.result} and proven in Section \ref{sec:proof.propositions}.

\section{The model}\label{sec:model}

    Let $(U_p, \, p ~\text{primes})$ be an i.i.d.\hspace{-0.3mm} sequence of uniform random variables on the unit circle in $\C$.
    The random field of interest is
    \begin{equation}\label{def:X}
        X(h) \circeq \sum_{p \leq T} W_p(h) \circeq \sum_{p \leq T} \frac{\text{Re}(U_p \, p^{-i h})}{p^{1/2}}, \quad h\in [0,1].
    \end{equation}
    (A sum over the variable $p$ always denotes a sum over primes.)
    This is a good model for the large values of $(\log |\zeta(\frac{1}{2} + i(\tau + h))|, h\in [0,1])$ for the following reason.
    Proposition 1 in \cite{arXiv:1304.0677} proves that, assuming the Riemann hypothesis, and for $T$ large enough, there exists a set $B\subseteq [T,T+1]$, of Lebesgue measure at least $0.99$, such that
    \begin{equation}
        \log |\zeta(\frac{1}{2} + i t)| = \text{Re}\left(\sum_{p \leq T} \frac{1}{p^{1/2 + it}} \frac{\log(T / p)}{\log T}\right) + O(1), \quad t\in B.
    \end{equation}
    If we ignore the smoothing term $\log(T / p) / \log T$ and note that the process $(p^{-i\tau}\hspace{-1mm}, p ~\text{primes})$, where $\tau$ is sampled uniformly in $[T,2T]$, converges, as $T\to\infty$ (in the sense of convergence of its finite-dimensional distributions), to a sequence of independent random variables distributed uniformly on the unit circle (by computing the moments), then the model \eqref{def:X} follows.
    For more information, see Section 1.1 in \cite{MR3619786}.

    More generally, for $-1 \leq r \leq k$, denote the increments of the field by
    \begin{equation}\label{def:X.r.k}
        X_{r,k}(h) \circeq \sum_{2^r < \log p \leq 2^k} \frac{\text{Re}(U_p \, p^{-i h})}{p^{1/2}}, \quad h\in [0,1].
    \end{equation}
    Differentiation of \eqref{def:X.r.k} yields
    \begin{equation}\label{def:X.prime.r.k}
        X_{r,k}'(h) = \sum_{2^r < \log p \leq 2^k} W_p'(h) = \sum_{2^r < \log p \leq 2^k} \frac{\text{Im}(U_p \, p^{-i h}) \log p}{p^{1/2}}.
    \end{equation}

\section{Main result}\label{sec:main.result}

    Throughout the paper, we will write $c$, $\widetilde{c}$, $c'$, and $c''$, for generic positive constants whose value may change at different occurrences.
    Here are the main side results of this paper.

    \begin{proposition}[Continuity estimates]\label{prop:continuity.estimates.derivative}
        Let $C > 0$.
        For any $-1 \leq r \leq k$, $0 \leq x \leq C (2^{2k} - 2^{2r})$, $2 \leq a \leq 2^{6k} - x$ and $h\in \R$,
        \begin{equation}\label{eq:prop:continuity.estimates.derivative}
            \prob{\max_{h':|h' - h| \leq 2^{-3k-1}} X_{r,k}'(h') \geq x + a, X_{r,k}'(h) \leq  x} \leq  c \exp\left(-2\frac{x^2}{2^{2k} - 2^{2r}} - \widetilde{c} \, a^{3/2}\right),
        \end{equation}
        where the constants $c$ and $\widetilde{c}$ only depend on $C$.
    \end{proposition}

    \begin{proposition}[Large deviation estimates]\label{prop:large.deviation.estimates.derivative}
        Let $C > 0$.
        For any $-1 \leq r \leq k$, $0 \leq x \leq C (2^{2k} - 2^{2r})$ and $h\in \R$,
        \begin{equation}\label{eq:prop:large.deviation.estimates.derivative}
            \prob{\max_{h':|h' - h| \leq 2^{-3k-1}} X_{r,k}'(h') \geq x} \leq c \exp\left(-2\frac{x^2}{2^{2k} - 2^{2r}}\right),
        \end{equation}
        where the constant $c$ only depends on $C$.
    \end{proposition}

    From the last proposition, we obtain the following theorem.

    \begin{theorem}[Main result]\label{thm:prop:large.deviation.estimates.derivative}
        Let $-1 \leq r \leq k$.
        For all $L > 0$, let $\mathcal{S}_{r,k,L}$ be a set of equidistant points in $[0,1]$ such that $|\mathcal{S}_{r,k,L}| = \lceil L \sqrt{2^{2k} - 2^{2r}} \sqrt{k \log 2}\rceil$ and $|h' - h| \geq |\mathcal{S}_{r,k,L}|^{-1}$ for different $h,h'\in \mathcal{S}_{r,k,L}$.
        Then, for any $K > 0$, there exists $L \circeq L(K) > 0$ large enough that
        \begin{equation}\label{eq:thm:prop:large.deviation.estimates.derivative}
            \PP\left(\Big|\max_{h\in[0,1]} X_{r,k}(h) - \max_{h\in \mathcal{S}_{r,k,L}} X_{r,k}(h)\Big| > K\right) < e^{-\frac{k}{4} (1 - e^{-K})^2 L^2}.
        \end{equation}
    \end{theorem}

    \begin{remark}\label{rem:thm:prop:large.deviation.estimates.derivative}
        When $r = -1$ and $2^k = \log T$, $X_{r,k}(h)$ is just the original model $X(h)$.
        In that case, \eqref{eq:thm:prop:large.deviation.estimates.derivative} shows that, with probability as close to $1$ as we want, there exists a discrete set $\mathcal{S} \subseteq [0,1]$ such that
        \begin{equation}\label{eq:reduce.to.discrete.problem}
            \max_{h\in [0,1]} X(h) - \max_{h\in \mathcal{S}} X(h) = O(1),
        \end{equation}
        where $|\mathcal{S}| = O(\log T \sqrt{\log \log T})$.
    \end{remark}

    We prove Theorem \ref{thm:prop:large.deviation.estimates.derivative} right away and we will prove Proposition \ref{prop:continuity.estimates.derivative} and Proposition \ref{prop:large.deviation.estimates.derivative} in Section \ref{sec:proof.propositions}.

    \begin{proof}[Proof of Theorem \ref{thm:prop:large.deviation.estimates.derivative}]
        For $M > 0$, define the event
        \begin{equation}
            E = \left\{\max_{h\in [0,1]} |X_{r,k}'(h)| \geq M \sqrt{2^{2k} - 2^{2r}} \sqrt{k \log 2}\right\}.
        \end{equation}
        Let $\mathcal{H}_k \circeq 2^{-3k} \Z$ and note that $|\mathcal{H}_k \cap [0,1]| = 2^{3k} + 1$.
        By a union bound, the symmetry of $X_{r,k}'(h)$'s distribution, and Proposition \ref{prop:large.deviation.estimates.derivative}, we obtain
        \begin{equation}\label{eq:thm:prop:large.deviation.estimates.derivative.eq.end.1}
            \PP(E) \leq \sum_{h\in \mathcal{H}_k \cap [0,1]} 2 \cdot \prob{\max_{h':|h' - h| \leq 2^{-3k-1}} X_{r,k}'(h') \geq M \sqrt{2^{2k} - 2^{2r}} \sqrt{k \log 2}} \leq (2^{3k} + 1) \cdot c \, 2^{-2k M^2}.
        \end{equation}

        For every realisation $\omega$ of the field $\{X_{r,k}(h)\}_{h\in [0,1]}$, let $h^{\star}(\omega)$ be a point where the maximum is attained.
        When $\omega\in E^c$, the mean value theorem yields that, for any $h(\omega)\in \mathcal{S}_{r,k,L}$ such that $|h^{\star}(\omega) - h(\omega)| \leq 2 / |\mathcal{S}_{r,k,L}|$, we have
        \begin{equation}\label{eq:thm:prop:large.deviation.estimates.derivative.compensation}
            e^{X_{r,k}(h^{\star}(\omega))} - e^{X_{r,k}(h(\omega))} = X_{r,k}'(\xi(\omega)) e^{X_{r,k}(\xi(\omega))} (h^{\star}(\omega) - h(\omega)) \leq \frac{2M}{L} e^{X_{r,k}(h^{\star}(\omega))},
        \end{equation}
        for some $\xi(\omega)$ lying between $h(\omega)$ and $h^{\star}(\omega)$. By taking $L \circeq L(K) \circeq 2 M / (1 - e^{-K})$, we deduce $e^{X_{r,k}(h(\omega))} \geq e^{-K} e^{X_{r,k}(h^{\star}(\omega))}$.
        This reasoning shows that, on the event $E^c$,
        \begin{equation}\label{eq:thm:prop:large.deviation.estimates.derivative.eq.end.2}
            \max_{h\in \mathcal{S}_{r,k,L}} X_{r,k}(h) \geq \max_{h\in [0,1]} X_{r,k}(h) - K.
        \end{equation}
        The conclusion follows from \eqref{eq:thm:prop:large.deviation.estimates.derivative.eq.end.2} and \eqref{eq:thm:prop:large.deviation.estimates.derivative.eq.end.1} with $M = \frac{1}{2} (1 - e^{-K}) L$.
    \end{proof}

\section{Proof of Proposition \ref{prop:continuity.estimates.derivative} and Proposition \ref{prop:large.deviation.estimates.derivative}}\label{sec:proof.propositions}

    We start by controlling the tail probabilities for a single point of the field's derivative.

    \begin{lemma}\label{lem:derivative.upper.bound}
        Let $C > 0$.
        For any $-1 \leq r \leq k$, $0 \leq x \leq C (2^{2k} - 2^{2r})$ and $h\in \R$,
        \begin{equation}\label{eq:lem:derivative.upper.bound}
            \prob{X_{r,k}'(h) \geq x} \leq c \exp\left(-2\frac{x^2}{2^{2k} - 2^{2r}}\right),
        \end{equation}
        where the constant $c$ only depends on $C$.
    \end{lemma}

    \begin{proof}
        Using Chernoff's inequality, the independence of the $U_p$'s and translation invariance, we have that, for all $\lambda \geq 0$,
        \begin{equation}\label{eq:lem:derivative.upper.bound.beginning}
            \prob{X_{r,k}'(h) \geq x} \leq e^{-\lambda x} \, \EE\big[e^{\lambda X_{r,k}'(h)}\big] = e^{-\lambda x} \prod_{2^r < \log p \leq 2^k} \EE\big[e^{\lambda W_p'(0)}\big].
        \end{equation}
        Note that
        \begin{equation}\label{eq:lem:derivative.upper.bound.beginning.2}
            \EE\big[e^{\lambda W_p'(0)}\big] = \frac{1}{2\pi} \int_0^{2\pi} \exp\left(\frac{\lambda \log p}{p^{1/2}} \sin(\theta)\right) d\theta = I_0\left(\frac{\lambda \log p}{p^{1/2}}\right),
        \end{equation}
        \cite[9.6.16, p.376]{MR0167642}, where $I_0$ denotes the {\it modified Bessel function of the first kind}.
        The function $I_0$ has the following series representation : $I_0(u) = 1 + \frac{u^2}{4} + \frac{u^4}{64} + O(u^6), ~u\in \R$.
        In turn,
        \begin{equation}\label{eq:expansion.log.I.0}
            \log(I_0(u)) = \frac{u^2}{4} - \frac{u^4}{64} + O(u^6), \quad u\in (-1,1),
        \end{equation}
        because $\log(1 + y) = y - \frac{y^2}{2} + O(y^3)$ for $y\in (-1,1)$, and $|I_0(u) - 1| < 1$ for $u\in (-1,1)$.
        Choose $\lambda = 4x / (2^{2k} - 2^{2r})$. By applying \eqref{eq:expansion.log.I.0} in \eqref{eq:lem:derivative.upper.bound.beginning.2}, the right-hand side of \eqref{eq:lem:derivative.upper.bound.beginning} is bounded from above by
        \begin{equation}\label{eq:lem:derivative.upper.bound.follow.up}
            c \, e^{-\lambda x} \exp\left(\sum_{2^r < \log p \leq 2^k} \frac{\lambda^2 (\log p)^2}{4 p} + \widetilde{c} \sum_{2^r < \log p \leq 2^k} \frac{\lambda^6 (\log p)^6}{p^3}\right).
        \end{equation}
        For the finite number of primes $p$ for which we cannot apply \eqref{eq:expansion.log.I.0} in \eqref{eq:lem:derivative.upper.bound.beginning.2} (note that $\lambda \log p < p^{1/2}$ holds for $p$ large enough since $\lambda \leq 4C$ by the assumption on $x$), the correction terms needed for \eqref{eq:lem:derivative.upper.bound.follow.up} to hold are absorbed in the constant $c$ in front of the first exponential in \eqref{eq:lem:derivative.upper.bound.follow.up}.
        The second sum in the big exponential is bounded by a constant independent from $r$ and $k$ since $\lambda \leq 4 C$ and $\sum_p (\log p)^6 p^{-3} < \infty$.
        By applying Lemma \ref{lem:technical.lemma} with $m = 2$, $\log P = 2^r$ and $\log Q = 2^k$, the first sum in the big exponential is bounded by $2x^2 / (2^{2k} - 2^{2r})$ up to an additive constant that only depends on $C$. The conclusion of the lemma follows.
    \end{proof}

    In the next lemma, we complement Lemma \ref{lem:derivative.upper.bound} by proving a large deviation estimate for $X'_{r,k}(0)$ and the difference $X_{r,k}'(h_2) - X_{r,k}'(h_1)$ jointly, where $|h_2 - h_1| \leq 2^{-3k}$.

    \begin{lemma}\label{lem:multivariate.derivative.chernoff}
        Let $C > 0$. For any $-1 \leq r \leq k$, $0 \leq x \leq C (2^{2k} - 2^{2r})$, $0 \leq y \leq 2^{6k}$, and any distinct $h_1,h_2\in \R$ such that $-2^{-3k-1} \leq h_1,h_2 \leq 2^{-3k-1}$,
        \begin{equation}\label{eq:lem:multivariate.derivative.chernoff}
            \prob{X_{r,k}'(0) \geq x, X_{r,k}'(h_2) - X_{r,k}'(h_1) \geq y} \leq c \exp\left(-2 \frac{x^2}{2^{2k} - 2^{2r}} - \frac{\widetilde{c} \, y^{3/2}}{|h_2 - h_1|\, 2^{3k}}\right),
        \end{equation}
        where the constants $c$ and $\widetilde{c}$ only depend on $C$.
    \end{lemma}

    \begin{proof}
        Assume that $y \geq \widetilde{C} |h_2 - h_1| 2^{3k}$ for a large constant $\widetilde{C} \geq 1$ because otherwise \eqref{eq:lem:multivariate.derivative.chernoff} follows from \eqref{eq:lem:derivative.upper.bound}.
        Since $|h_2 - h_1| 2^{3k} \leq 1$, note that this assumption also implies $y^{1/2} \geq \widetilde{C}^{1/2} |h_2 - h_1| 2^{3k}$.
        For all $\lambda_1,\lambda_2 \geq 0$, the left-hand side of \eqref{eq:lem:multivariate.derivative.chernoff} is bounded from above (using Chernoff's inequality) by
        \begin{equation}\label{eq:lem:multivariate.derivative.chernoff.beg.1}
            \esp{\exp(\lambda_1 X_{r,k}'(0) + \lambda_2 (X_{r,k}'(h_2) - X_{r,k}'(h_1)))} \exp(-\lambda_1 x - \lambda_2 y).
        \end{equation}
        We will show that if $0 \leq \lambda_1 \leq 4C$ and $0 \leq \lambda_2 \leq |h_2 - h_1|^{-1}$, then
        \begin{equation}\label{eq:lem:multivariate.derivative.chernoff.beg.2}
            \begin{aligned}
                &\esp{\exp(\lambda_1 X_{r,k}'(0) + \lambda_2 (X_{r,k}'(h_2) - X_{r,k}'(h_1)))} \\
                &\quad\leq c\, \exp\left(\frac{\lambda_1^2}{8} (2^{2k} - 2^{2r}) + c\, \lambda_2 |h_2 - h_1| \, 2^{3k} + c^2 \lambda_2^2 |h_2 - h_1|^2 \, 2^{4k}\right).
            \end{aligned}
        \end{equation}
        The result \eqref{eq:lem:multivariate.derivative.chernoff} follows by choosing $\lambda_1 = 4x / (2^{2k} - 2^{2r})$, $\lambda_2 = y^{1/2} |h_2 - h_1|^{-1}\, 2^{-3k}$ and $\widetilde{C}$ large enough (with respect to $c$) in \eqref{eq:lem:multivariate.derivative.chernoff.beg.1} and \eqref{eq:lem:multivariate.derivative.chernoff.beg.2}.
        The assumptions on $x$, $y$, $h_1$ and $h_2$ ensure that $0 \leq \lambda_1 \leq 4C$ and $0 \leq \lambda_2 \leq |h_2 - h_1|^{-1}$.
        We now prove \eqref{eq:lem:multivariate.derivative.chernoff.beg.2}. For $2^r < \log p \leq 2^k$, the quantity
        \begin{equation}\label{eq:lem:multivariate.derivative.chernoff.beg.w}
            \esp{\exp(\lambda_1 W_p'(0) + \lambda_2(W_p'(h_2) - W_p'(h_1)))}
        \end{equation}
        (recall $W_p'(h)$ from \eqref{def:X.prime.r.k}) can be written as
        \begin{equation}
            \frac{1}{2\pi}\int_0^{2\pi} \hspace{-2mm}\exp\left(\frac{\log p}{p^{1/2}} \Big\{\lambda_1\sin\theta + \lambda_2(\sin(\theta - h_2 \log p) - \sin(\theta - h_1\log p))\Big\}\right) d\theta.
        \end{equation}
        Since $\sin(\theta - \eta) = \sin(\theta)\cos(\eta)- \cos(\theta)\sin(\eta)$ and
        \begin{equation}
            \frac{1}{2\pi}\int_0^{2\pi} \hspace{-1mm}\exp(a \cos\theta + b \sin\theta) d\theta = I_0(\sqrt{a^2 + b^2}),
        \end{equation}
        \cite[9.6.16, p.376]{MR0167642}, then \eqref{eq:lem:multivariate.derivative.chernoff.beg.w} is equal to
        \begin{equation}\label{eq:lem:multivariate.derivative.chernoff.beg.I.0}
            I_0\left(\sqrt{\frac{(\log p)^2}{p}
                \left\{\hspace{-1mm}
                \begin{array}{l}
                    \big(\lambda_1 + \lambda_2(\cos(h_2 \log p) - \cos(h_1 \log p))\big)^2 \\
                    + \big(\lambda_2(\sin(h_1 \log p) - \sin(h_2 \log p))\big)^2
                \end{array}
                \hspace{-1mm}\right\}}\right).
        \end{equation}
        From \eqref{eq:expansion.log.I.0}, note that
        \begin{equation}\label{eq:lem:multivariate.derivative.chernoff.beg.asymptotics}
            \log(I_0(\sqrt{u})) = \frac{u}{4} - \frac{u^2}{64} + O(u^3), \quad u\in (-1,1).
        \end{equation}
        Also, note that
        \begin{equation}\label{eq:lem:multivariate.derivative.chernoff.sin.asymptotics}
            \begin{aligned}
                &\sin(h_1 \log p) - \sin(h_2 \log p) = O(|h_2 - h_1| \log p), \\
                &\cos(h_2 \log p) - \cos(h_1 \log p) = O(|h_2 - h_1| \log p).
            \end{aligned}\vspace{0.7mm}
        \end{equation}
        If we put \eqref{eq:lem:multivariate.derivative.chernoff.beg.w}, \eqref{eq:lem:multivariate.derivative.chernoff.beg.I.0}, \eqref{eq:lem:multivariate.derivative.chernoff.beg.asymptotics} and \eqref{eq:lem:multivariate.derivative.chernoff.sin.asymptotics} together, we get, for $p$ large enough,
        \begin{align}\label{eq:lem:multivariate.derivative.chernoff.eq.end}
            \log \eqref{eq:lem:multivariate.derivative.chernoff.beg.w}
            &\leq \frac{(\log p)^2}{4p} \left\{\big(\lambda_1 + c\, \lambda_2 |h_2 - h_1| \log p\big)^2 + \big(c\, \lambda_2 |h_2 - h_1| \log p)\big)^2\right\} + \frac{\widetilde{c}}{p^2} \notag \\
            &\leq \frac{\lambda_1^2}{4} \frac{(\log p)^2}{p} + c\, \lambda_2 |h_2 - h_1| \frac{(\log p)^3}{p} + c^2 \lambda_2^2 |h_2 - h_1|^2 \frac{(\log p)^4}{p} + \frac{\widetilde{c}}{p^2}.
        \end{align}
        To obtain the last inequality, we used the fact that $\lambda_1 \leq 4C$.
        After summing \eqref{eq:lem:multivariate.derivative.chernoff.eq.end} over $2^r < \log p \leq 2^k$ and using Lemma \ref{lem:technical.lemma}, we deduce
        \begin{equation}
            \begin{aligned}
                &\log \esp{\exp(\lambda_1 X_{r,k}'(0) + \lambda_2(X_{r,k}'(h_2) - X_{r,k}'(h_1)))} \\[2mm]
                &\hspace{10mm}\leq \widetilde{c} + \frac{\lambda_1^2}{8} (2^{2k} - 2^{2r}) + c\, \lambda_2 |h_2 - h_1| \, 2^{3k} + c^2 \lambda_2^2 |h_2 - h_1|^2 \, 2^{4k},
            \end{aligned}
        \end{equation}
        where the constants $c$ and $\widetilde{c}$ only depend on $C$.
        This is exactly \eqref{eq:lem:multivariate.derivative.chernoff.beg.2}.
    \end{proof}

    We are now ready to prove Proposition \ref{prop:continuity.estimates.derivative}.
    For $k\in \N_0$, recall that $\mathcal{H}_k \circeq 2^{-3k} \Z$, so that $\mathcal{H}_0 \subseteq \mathcal{H}_1 \subseteq \ldots \subseteq \mathcal{H}_k \subseteq \ldots \subseteq \R$ is a nested sequence of sets of equidistant points and $|\mathcal{H}_k \cap [0,1)| = 2^{3k}$.

    \begin{proof}[Proof of Proposition \ref{prop:continuity.estimates.derivative}]
        Without loss of generality, we may assume $h = 0$.
        We can also round $x$ up to the nearest larger integer and decrease $a$ so that we may assume that $x\in \N_0$ and $a \geq 1$.
        To see why this is possible, define the new values of $x$ and $a$ by $\widetilde{x} \circeq \lceil x \rceil$ and $\widetilde{a} \circeq a - \widetilde{x} + x$, respectively. Since $x + a = \widetilde{x} + \widetilde{a}$ and $x \leq \widetilde{x}$, and assuming that we can show \eqref{eq:prop:continuity.estimates.derivative} with $\widetilde{x}$ and $\widetilde{a}$, we would have
        \begin{equation}
            \begin{aligned}
            \prob{\max_{h':|h' - h| \leq 2^{-3k-1}} X_{r,k}'(h') \geq x + a, X_{r,k}'(h) \leq  x}
            &\leq \prob{\max_{h':|h' - h| \leq 2^{-3k-1}} X_{r,k}'(h') \geq \widetilde{x} + \widetilde{a}, X_{r,k}'(h) \leq  \widetilde{x}} \\
            &\leq c ~ \exp\left(-2\frac{\widetilde{x}^2}{2^{2k} - 2^{2r}} - \widetilde{c} ~ \widetilde{a}^{3/2}\right) \\
            &\leq c' \exp\left(-2\frac{x^2}{2^{2k} - 2^{2r}} - c'' a^{3/2}\right),
            \end{aligned}
        \end{equation}
        where the constants $c'$ and $c''$ only depend on $C$.

        It remains to show \eqref{eq:prop:continuity.estimates.derivative} when $x\in \N_0$ and $a \geq 1$.
        We choose to adapt the chaining argument found in \cite[Proposition 2.5]{MR3619786}.
        Define the events
        \begin{equation}
            B_x \circeq \{X_{r,k}'(0) \leq 0\} \quad \text{and} \quad B_q \circeq \{X_{r,k}'(0)\in [x - q - 1,x - q]\}, \quad q\in \{0,1,\ldots,x-1\}.
        \end{equation}
        Note that the left-hand side of \eqref{eq:prop:continuity.estimates.derivative} is at most
        \begin{equation}\label{eq:prop:continuity.estimates.derivative.eq.beginning}
            \sum_{q=0}^x \PP\left(B_q \cap \Big\{\max_{h'\in A} \{X_{r,k}'(h') - X_{r,k}'(0)\} \geq a + q\Big\}\right),
        \end{equation}
        where $A = [-2^{-3k-1},2^{-3k-1}]$.
        Let $(h_i, i\in \N_0)$ be a sequence such that $h_0 = 0$, $h_i\in \mathcal{H}_{k+i} \cap A$, $\lim_{i\to\infty} h_i = h'$ and $|h_{i+1} - h_i|\in\{0,\frac{1}{8} 2^{-3(k+i)}, \frac{2}{8} 2^{-3(k+i)}, \frac{3}{8} 2^{-3(k+i)}, \frac{4}{8} 2^{-3(k+i)}\}$ for all $i$.
        Because the map $h\mapsto X_{r,k}'(h)$ is almost-surely continuous,
        \begin{equation}
            X_{r,k}'(h') - X_{r,k}'(0) = \sum_{i=0}^{\infty} (X_{r,k}'(h_{i+1}) - X_{r,k}'(h_i)).
        \end{equation}
        Since $\sum_{i=0}^{\infty} \frac{1}{2 (i+1)^2} \leq 1$, we have the inclusion of events,
        \begin{equation}
            \left\{X_{r,k}'(h') - X_{r,k}'(0) \geq a + q\right\}
            \subseteq \bigcup_{i=0}^{\infty} \left\{X_{r,k}'(h_{i+1}) - X_{r,k}'(h_i) \geq \frac{a + q}{2 (i+1)^2}\right\}.
        \end{equation}
        This implies that $\{\max_{h'\in A} X_{r,k}'(h') - X_{r,k}'(0) \geq a + q\}$ is included in
        \begin{equation}
            \bigcup_{i=0}^{\infty} \bigcup_{\substack{h_1\in \mathcal{H}_{k+i} \cap A \\ |h_2 - h_1| = \frac{j}{8} 2^{-3(k + i)}  \\ \text{for some } j\in \{1,2,3,4\}}} \left\{X_{r,k}'(h_2) - X_{r,k}'(h_1) \geq \frac{a + q}{2 (i+1)^2}\right\},
        \end{equation}
        where we have ignored the case $h_1 = h_2$ since the event $\{X_{r,k}'(h_2) - X_{r,k}'(h_1) \geq \frac{a + q}{2 (i+1)^2}\}$ is the empty set.
        Because $|\mathcal{H}_{k+i} \cap A| \leq c \, 2^{3i}$, the $q$-th summand in \eqref{eq:prop:continuity.estimates.derivative.eq.beginning} is at most,
        \begin{equation}\label{eq:prop:continuity.estimates.derivative.eq.middle}
            \sum_{i=0}^{\infty} c \, 2^{3i} \sup_{\substack{h_1\in \mathcal{H}_{k+i}\cap A \\ |h_2 - h_1| = \frac{j}{8} 2^{-3(k + i)}  \\ \text{for some } j\in \{1,2,3,4\}}} \PP\left(B_q \cap \left\{X_{r,k}'(h_2) - X_{r,k}'(h_1) \geq \frac{a + q}{2 (i+1)^2}\right\}\right).
        \end{equation}
        Note that $a + q \leq a + x \leq 2^{6k}$ by assumption.
        Lemma \ref{lem:multivariate.derivative.chernoff} can thus be applied to get that \eqref{eq:prop:continuity.estimates.derivative.eq.middle} is at most
        \begin{equation}
            c \sum_{i=0}^{\infty} 2^{3i} \exp\left(-2\frac{(x - q - 1)^2}{2^{2k} - 2^{2r}} - \widetilde{c} ~ 2^{3i} \frac{(a + q)^{3/2}}{(i+1)^3}\right) \leq c' e^{-2 \frac{(x - q - 1)^2}{2^{2k} - 2^{2r}} - \widetilde{c} (a + q)^{3/2}}.
        \end{equation}
        Since $e^{-\widetilde{c} (a+q)^{3/2}} \leq e^{-\widetilde{c} a^{3/2} - \widetilde{c} q^{3/2}}$, \eqref{eq:prop:continuity.estimates.derivative.eq.beginning} is at most
        \begin{equation}\label{eq:prop:continuity.estimates.derivative.eq.end}
            \begin{aligned}
            c' \, e^{-\widetilde{c} a^{3/2}} \sum_{q=0}^x e^{-2 \frac{(x - q - 1)^2}{2^{2k} - 2^{2r}} - \widetilde{c} q^{3/2}}
            &\leq c' \, e^{-\frac{2 x^2}{2^{2k} - 2^{2r}} - \widetilde{c} a^{3/2}} \sum_{q=0}^x e^{4C (q + 1) - \widetilde{c} q^{3/2}} \\
            &\leq c'' e^{-\frac{2 x^2}{2^{2k} - 2^{2r}} - \widetilde{c} a^{3/2}},
            \end{aligned}
        \end{equation}
        where we used the assumption $x \leq C (2^{2k} - 2^{2r})$ to obtain the first inequality in \eqref{eq:prop:continuity.estimates.derivative.eq.end}. This proves \eqref{eq:prop:continuity.estimates.derivative}.
    \end{proof}

    \begin{proof}[Proof of Proposition \ref{prop:large.deviation.estimates.derivative}]
        The left-hand side of \eqref{eq:prop:large.deviation.estimates.derivative} is at most
        \begin{equation}
            \prob{X_{r,k}'(h) \geq x - 2} + \PP\left(\hspace{-1mm}
            \begin{array}{l}
                \max_{h':|h' - h| \leq 2^{-3k-1}} X_{r,k}'(h') \geq (x-2) + 2, \\[1mm]
                X_{r,k}'(h) \leq  x - 2
            \end{array}
            \hspace{-1mm}\right)
        \end{equation}
        The conclusion follows from Lemma \ref{lem:derivative.upper.bound} and Proposition \ref{prop:continuity.estimates.derivative} with $x-2$ in place of $x$ and $a=2$.
    \end{proof}

\appendix
\section{Technical lemma}

    \begin{lemma}\label{lem:technical.lemma}
        Let $m \geq 1$ and $1 \leq P < Q$, then
        \begin{equation}
            \bigg|\sum_{P < p \leq Q} \frac{(\log p)^m}{p} - \left(\frac{(\log Q)^m}{m} - \frac{(\log P)^m}{m}\right)\bigg| \leq D,
        \end{equation}
        where $D > 0$ is a constant that only depends on $m$.
    \end{lemma}

    \begin{proof}
        Without loss of generality, assume that $P \geq 2$.
        We use a standard form of the prime number theorem \cite[Theorem 6.9]{MR2378655} which states that
        \begin{equation}\label{eq:prime.number.theorem}
            \#\{p ~\text{prime} : p \leq x\} = \int_2^x \frac{1}{\log u} du + R(x),
        \end{equation}
        where $R(x) = O(x e^{-c \sqrt{\log x}})$, uniformly for $x \geq 2$.
        Using \eqref{eq:prime.number.theorem} and integration by parts, we have
        \begin{equation}\label{eq:lem:Y.j.prime.covariance.estimates.variance.calculations}
            \begin{aligned}
            \sum_{P < p \leq Q} \frac{(\log p)^m}{p}
            &= \int_P^Q \frac{(\log u)^{m-1}}{u} du + \int_P^Q \frac{(\log u)^m}{u} d R(u) \\
            &= \frac{(\log Q)^m}{m} - \frac{(\log P)^m}{m} + \frac{(\log Q)^m}{Q} R(Q) - \frac{(\log P)^m}{P} R(P) \\
            &\quad - \int_P^Q \frac{(m - \log u) (\log u)^{m-1}}{u^2} R(u) du.
            \end{aligned}
        \end{equation}
        By making the change of variable $z = c \sqrt{\log u}$ on the right-hand side of \eqref{eq:lem:Y.j.prime.covariance.estimates.variance.calculations}, note that
        \begin{align}
            \bigg|\int_P^Q \frac{(m - \log u) (\log u)^{m-1}}{u^2} R(u) du\bigg| \leq \widetilde{D} \int_0^{\infty} z^{2m + 1} e^{-z} dz = \widetilde{D} \, \Gamma(2m + 2),
        \end{align}
        where $\widetilde{D} > 0$ is a constant that only depends on $m$. This ends the proof.
    \end{proof}

\section*{Acknowledgements}

We would like to thank the anonymous referee for his valuable comments that led to improvements in the presentation of this paper.

%
%


\bibliographystyle{authordate1}
\bibliography{Arguin_Ouimet_2018_new_derivatives_random_zeta_bib}

\end{document}